\documentclass[reqno]{llncs}
\usepackage{amsmath}
\usepackage{amsfonts}
\usepackage{amssymb}
\usepackage{blkarray}
\usepackage{mathtools}
\usepackage{multirow}
\usepackage{bbold}

\newcommand{\R}{\mathbb{R}}

\newcommand{\x}{\textbf{x}}

\newcommand{\w}{\textbf{w}}
\newcommand{\z}{\textbf{z}}

\usepackage{todonotes}

\usepackage{algorithm}
\usepackage[noend]{algpseudocode}
\usepackage{tikz}

\usepackage[margin=1.5in]{geometry}

\usepackage[colorlinks=true,breaklinks=true,bookmarks=false,urlcolor=blue,citecolor=blue,linkcolor=blue,bookmarksopen=false,draft=false]{hyperref}

\begin{document}

\title{An Integer Program for Pricing Support Points of Exact Barycenters}

\author{Steffen Borgwardt\inst{1} \and Stephan Patterson\inst{2}}

\institute{\email{\href{mailto:steffen.borgwardt@ucdenver.edu}{steffen.borgwardt@ucdenver.edu}};
University of Colorado Denver 
\and \email{\href{mailto:stephan.patterson@lsus.edu}{stephan.patterson@lsus.edu}}; Louisiana State University in Shreveport
}

\date{}

\maketitle
\begin{abstract}
The computation of exact barycenters for a set of discrete measures is of interest in applications where sparse solutions are desired, and to assess the quality of solutions returned by approximate algorithms and heuristics. The task is known to be NP-hard for growing dimension and, even in low dimensions, extremely challenging in practice due to an exponential scaling of the linear programming formulations associated with the search for sparse solutions.

A common approach to facilitate practical computations is an approximation based on the choice of a small, fixed set $S_0$ of support points, or a fixed set $S^*_0$ of combinations of support points from the measures, that may be assigned mass. Through a combination of linear and integer programming techniques, we model an integer program to compute additional combinations, and in turn support points, that, when added to $S^*_0$ or $S_0$, allow for a better approximation of the underlying exact barycenter problem. The approach improves on the scalability of a classical column generation approach: 
instead of a pricing problem that has to evaluate exponentially many reduced cost values, we solve a mixed-integer program of quadratic size. The properties of the model, and practical computations, reveal a tailored branch-and-bound routine as a good solution strategy.
\end{abstract}


\section{Introduction}

Barycenter problems have seen much attention in the literature due to their relevance for applications in a wide range of fields; see Section \ref{sec:literature}. Data for these problems is typically represented as so-called {\em discrete (probability) measures} $P_i$, i.e., probability measures with a finite number $|P_i|$ of support points.  

The input for the {\em discrete barycenter problem} is a set of discrete measures $P_1, \ldots, P_n$ whose support points lie in $\R^d$ and a corresponding set of positive weights $\lambda_1, \ldots, \lambda_n$ with $\sum_{i=1}^n \lambda_i = 1$. Then the discrete barycenter problem is: find a measure $\bar P$ such that
\[ \phi(\bar P) := \sum\limits_{i=1}^n \lambda_i W_2(\bar P, P_i)^2 = \inf\limits_{P \in \mathcal{P}^2(\R^d)} \sum\limits_{i=1}^n \lambda_i W_2(P,P_i)^2,\]
where $W_2$ is the quadratic Wasserstein distance and $\mathcal{P}^2(\R^d)$ is the set of all probability measures on $\R^d$ with finite second moments \cite{ac-11}. Because the measures $P_1, \ldots, P_n$ are discrete, all barycenters $\bar P$ are also supported on a finite subset of $\R^d$. In addition to information on the support set and associated masses, computing the associated transport cost $\phi(\bar P )$ requires identifying the {\em destination (support) points} in each $P_1, \ldots, P_n$ to which mass is transported from each support point of $\bar P$. Then the optimal transport cost $\phi (\bar P)$ can be computed as the summation of the weighted squared Euclidean distances between each barycenter support point and its destination points \cite{abm-16,coo-15,v-09}. 


\subsection{Fixed-Support Approximations}\label{sec:literature}

The computation of barycenters arises in a vast range of applications, from statistics \cite{mtbmmh-15,zp-17}, over economics \cite{bhp-13,cmn-10}, game theory \cite{ce-10,coo-15}, physics and the material sciences \cite{cfk-13,jzd-98,ty-05}, to image processing \cite{jbtcg-19,sa-19,cad-20} and machine learning \cite{hbccp-19,hhpj-19,hnybhp-17,shbmccps-18,ydpbbgc-19}. We refer the reader to the recent surveys \cite{pc-18,pz-19} and the textbook \cite{v-09}. 

In recent years, there has been growing interest in approaching clustering and learning applications for high-dimensional data \cite{cad-20,hhpj-19,hnybhp-17,shdj-20}, such as those arising in natural language processing \cite{xwlc-18,ylwwll-17}. 
However, the computation of a barycenter or an approximation is challenging: the barycenter problem is known to be NP-hard for growing dimension, and this hardness extends to approximate computations, some special cases, and other optimal transport metrics \cite{ab-21}. A variation in which one wants to find a sparse barycenter, i.e., one with a low number of support points, remains NP-hard even for dimension two and three measures \cite{bp-21a}. In contrast, the barycenter problem is solvable in polynomial time for fixed dimension \cite{ab-21b}. 
However, even for dimension $d=2$, this observation does not lead to a practical algorithm.


There is a wide field of research that addresses these challenges through approximate and heuristic methods. A popular approach is to a priori specify a support set $S_0\subset \mathbb{R}^d$ of {\em possible support points} and then optimize over measures supported on $S_0$, i.e., over the masses associated to each support point. There are many variations of such {\em fixed-support approximations} \cite{ab-21}. 
In particular, a restriction of $S_0$ to the union of original support points in the measures leads to a strongly-polynomial-time $2$-approximation in any dimension \cite{b-17}. For $n$ measures supported on a shared grid, a barycenter is contained in an $n$-times finer grid \cite{bp-18}. 

In general, the fixation of a polynomial-size $S_0$ leads to a polynomial-size linear program (LP) (and avoids an exponential scaling in $d$) which can be solved directly using blackbox LP solvers or through specially tailored approaches such as entropic regularization introduced by \cite{c-13,cd-14}. In the latter approach, the LP is made strongly convex through the addition of a small entropy cost, so that more efficient solvers can be applied. This approach has been refined for competitive, approximate large-scale computations; see, e.g., \cite{bccnp-14,jcg-20,ktddgu-19,lhccj-20,sgpcbndg-15}

An alternative approach to obtain an LP is to use the {\em non-mass splitting property} \cite{abm-16} of barycenters and the Wasserstein distance: given a set of single support points in $n$ measures, the optimal location for joint transport of mass to them is their weighted mean; we describe this observation in some detail in Section \ref{sec:reducedcosts}. A barycenter decomposes into such {\em combinations of support points}, each associated some mass, of the original measures. In turn, this allows one to perform an exact barycenter computation through modeling an LP that assigns mass optimally to the set $S^*$ of all combinations; see, e.g., \cite{abm-16,bccnp-14,bp-21}. Note that $S^*$ generally is of exponential size $|S^*|=\prod_{i=1}^n |P_i|$. 
Again, it is of interest to approximate a barycenter through a smaller set of {\em possible combinations} $S_0^*$. 

Recent algorithms that take alternative approaches, i.e., do not build on $S_0$ or $S_0^*$, such as the Frank-Wolfe algorithm of \cite{lspc-10} and the Functional Gradient Descent algorithm of \cite{swrh-20} are rare exceptions. They underline a desire to step away from a fixation of $S_0$ or $S_0^*$; at the same time, it remains desirable to retain the ability to connect to and use tools from the long line of research on fixed-support approximations.


\subsection{Contributions and Outline}\label{sec:contributions}

For both types of fixed-support approximations, based on optimization over a given support set $S_0$ or a set of combinations $S_0^*$, the size of $S_0$ or $S_0^*$ is the bottleneck. In this work, we are interested in the generation of support points or combinations that can be used for $S_0$ or $S_0^*$, respectively.  
We describe an integer program, and develop a branch-and-bound approach, to find a new combination to add to a given set $S_0^*$ that could allow an improvement of an optimal barycenter approximation over the new $S_0^*$. The combination corresponds to a new support point for $S_0$, too. As such, our method can be used to, {\em a posteriori}, improve an approximate barycenter computed through {\em any} method in the literature.

When used in an iterative scheme, our method would improve any approximate barycenter to an {\em exact} barycenter. Thus, it can be used to compete with the method in \cite{bp-21} and will outperform it as soon as the exponential-size reduced-cost vector becomes prohibitive compared to setup and solution of a polynomial-size integer program. As memory requirements remain low, and only computation time spent remains as a main bottleneck, one obtains the ability to compute exact barycenters for larger instances than before. 

However, this favorable behavior also reveals its main challenges: by design, the method (if run in an iterative scheme) solves the NP-hard exact barycenter problem \cite{ab-21}. In turn, the approach is computationally expensive and can only be expected to be applied to problems of small-to-moderate size. Further, multiple combinations may have to be computed before a strict improvement is possible and accuracy of computations requires special attention.

The paper is structured as follows. In Section \ref{sec:reducedcosts}, we first recall some background on an (exponential-size) LP formulation for the barycenter problem based on $S^*$ and the non-mass splitting property. Then we turn to our main contribution: we devise an integer program (IP) to search for a combination of support points to add to a given $S_0^*$ such that an improved approximation may be possible. We then show that this IP can be transformed into a simpler mixed-integer program (MIP) through the use of some optimality conditions.

In Section \ref{sec:IpvsLP}, we discuss a practical implementation. We first turn to a theoretical study of properties of the mixed-integer program, such as the fractionality of optimal solutions in a relaxation, which inform the design of solution strategies. Then we analyze the practical results of computational experiments, and discuss the potential and limitations of the proposed approach. The experiments highlight a trade-off of lower memory requirements with higher setup times and reveal viable approaches in practice.
We conclude in Section \ref{sec:conclusion} with some final remarks and promising directions of research for further improvements.

\section{An Integer Program for Support Point Generation}\label{sec:reducedcosts}

\subsection{A Linear Program for Exact Barycenter Computations}\label{sec:barylp}

We begin by recalling the construction of a linear program for the computation of a barycenter based on the set $S^*$ of all combinations of support points of measures $P_1, \ldots, P_n$. It has been shown \cite{abm-16} that all barycenters satisfy the {\em non-mass-splitting property}, that is, all mass associated with a barycenter support point is transported to a single destination point $\x_i$ in each $P_i$, $i = 1, \ldots, n$. Furthermore, for a given combination of destination points, the mass assigned to that combination must be located at the weighted mean $\sum_{i=1}^n \lambda_i \x_i$, as any other location produces a strictly increased cost. For this reason, a barycenter can be described as a set of combinations of destination points and corresponding masses, and the location of the mass assignment can be computed as the weighted mean, when needed. Specifically, 
$S^*$ contains elements $s_h = (\x_1^h, \ldots, \x_n^h)$ for $h = 1, \ldots, \prod_{i=1}^n |P_i|$, and each $s_h$ has a corresponding fixed unit-mass transport cost $c_h$ for each potential combination of destination points. The cost $c_h$  be computed without computing the weighted mean itself, shown in \cite{bp-21} to be
\[c_h = \sum_{i=1}^{n-1} \lambda_i \sum_{j=i+1}^{n} \lambda_j || \x_i^h - \x_j^h ||^2. \] 
Then the transport cost of the barycenter $\phi(\bar P)$ can be computed as the sum of the mass assigned to $s_h$ times the unit transport cost  $c_h$. 

A resulting linear program for solving the discrete barycenter problem is as follows. Let $c=(c_h)$ be the vector stating the unit-mass transport costs for each element of $S^*$, and let $d=(d_{ik})$ be the vector stating the mass at support point $\x_{ik}$, the $k^{th}$ support point of $P_i$. The $0,1$-matrix $A$ gives constraints that enforce that each support point in each measure receives the correct mass: each row has $1$-entries for all combinations $s_h$ in which $\x_{ik}$ appears, and $0$ otherwise. The variable vector $\w=(w_h)$ denotes mass assigned to each combination $s_h$. This gives a standard form linear program.
\begin{equation*}\label{LPw}
\begin{array}{crl}
\tag{Bary-LP}  \mathrm{min}&  c^T \w & \nonumber \\
\mathrm{s.t.  } &A \w = &d \\
&\w \geq & 0. \nonumber
\end{array}
\end{equation*}

The number of variables in LP (\ref{LPw}) is $\prod_{i=1}^{n} |P_i|$, which scales exponentially in the number of measures $n$. However, LP (\ref{LPw}) contains only $\sum_{i=1}^n |P_i|$ constraints, which makes it a prime candidate for column generation as explored in \cite{bp-21}. 

\subsection{An Integer Program for Support Point Generation}

Classical column generation begins with a (reduced) master problem containing only a small number of all possible variables; for LP (\ref{LPw}), this is a small number of all possible combinations. Information from the current solution is used to choose a new variable, found by solving a separate pricing problem, for introduction to the master problem.

Specifically, a reduced master problem begins with a set of possible combinations $S_0^*$ that is a small subset of the set of all combinations $S^*$. 
Then, the classic pricing problem uses the dual vector $y$ from the current master problem, whose elements correspond to each constraint in the master problem, and $A_h$, the column of $A$ corresponding to each $s_h$. By minimizing $c_h - y^T A_h$ (or, equivalently, maximizing $y^T A_h- c_h$), an index $h^*$ corresponding to combination $s_{h^*}$ is found, and $s_{h^*}$ is the chosen combination for introduction to the master problem. 

It is possible to solve the classic pricing problem by fully processing the vector $c$, which has one entry for each combination $s_h \in S^* \backslash S_0^*$, and selecting the best value \cite{bp-21}. The primary goal and contribution of the work in this paper is to reformulate the pricing problem {\em without} searching the exponentially-scaling $c$. Instead, we seek to choose a combination $s_{h^*}$ by formulating a pricing program that will individually select an element $\x_i$ of each measure $P_i$. We begin the development of a new formulation by expanding the expression for $c_h$ from Section \ref{sec:barylp} in the same manner as in \cite{bp-21}:
\begin{align*} c_h &= \sum_{i=1}^{n-1} \lambda_i \sum_{j=i+1}^{n} \lambda_j (\x_i^h - \x_j^h)^T(\x_i^h - \x_j^h) \\
 &= \sum_{i=1}^{n-1} \lambda_i \sum_{j=i+1}^n \lambda_j ((\x_i^h)^T \x_i^h + (\x_j^h)^T \x_j^h) - 2\sum_{i=1}^{n-1}\lambda_i (\x_i^h)^T \sum_{j=i+1}^n \lambda_j \x_j^h \\
&= \sum_{i=1}^{n-1}\sum_{j=i+1}^{n}\lambda_i \lambda_j ((\x_i^h)^T \x_i^h -2(\x_i^h)^T \x_j^h +(\x_j^h)^T \x_j^h).
\end{align*}

We reformulate this cost expression, returning to using the original list of support points $\x_{ik}$ instead of the combination $h$. Additionally, instead of the classical pricing objective function $\max y^T A_h- c_h$, we use $\text{max} \sum_{h} y^T A_h -c_h$ and will guarantee through a set of constraints that a single index from $h$ is selected and all other summands are $0$.  

For each $i$, let $z_{ik}$ be a $0,1$-variable where we interpret $z_{ik} = 1$ to mean $\x_i^h = \x_{ik}$. Then the first term $(\x_i^h)^T \x_i^h$ of the expansion above can be rewritten as $\sum_{i=1}^{n-1} \sum_{j=i+1}^{n} \lambda_i \lambda_j (\x_i^h)^T \x_i^h = \sum_{i=1}^{n-1}\sum_{j=i+1}^{n}\sum_{k=1}^{|P_i|}\lambda_i \lambda_j \x_{ik}^T \x_{ik} z_{ik}$, with the other terms being reformulated similarly.   

Additionally, at the optimal index $h$, the vector $\z = (z_{ik})$ is precisely $A_h$, and so $y^T A_h$ may reformulated as well. Since the constraints in LP (\ref{LPw}) ensure that $\x_{ik}$, the $k^{th}$ support point of measure $P_i$, receives the correct mass $d_{ik}$, we match the formatting on the index of the elements of the dual vector $y$ as $y_{ik}$, $i = 1,\ldots,n$ and $k = 1, \ldots, |P_i|$, giving $y=(y_{ik})$, and $y^T A_h = \sum_{i=1}^n \sum_{k=1}^{|P_i|} y_{ik}z_{ik}$.
\[ \text{max} \;\sum_{i=1}^n \sum_{k=1}^{|P_i|} y_{ik}z_{ik} - \sum_{i=1}^{n-1}\sum_{j=i+1}^{n} \sum_{k=1}^{|P_i|} \sum_{m=1}^{|P_j|} \lambda_i \lambda_j (\x_{ik}^T \x_{ik} -2\x_{ik}^T \x_{jm} +\x_{jm}^T \x_{jm})z_{ik}z_{jm}.\]

To enforce that we choose exactly one element from each measure, we introduce constraints
\[ \sum_{k = 1}^{|P_i|} z_{ik} = 1 \hspace{.2in} \forall i = 1, \ldots, n. \]
Satisfaction of these constraints allows us to expand the quadruple-sum in the objective and simplify two of its three terms to  
\begin{align*} & \text{max}  \; \sum_{i=1}^n\sum_{k=1}^{|P_i|} y_{ik}z_{ik} - \sum_{i=1}^{n-1}\sum_{j=i+1}^{n} \sum_{k=1}^{|P_i|} \lambda_i \lambda_j \x_{ik}^T \x_{ik}z_{ik} \\
&  +2\sum_{i=1}^{n-1}\sum_{j=i+1}^{n} \sum_{k=1}^{|P_i|} \sum_{m=1}^{|P_j|}\lambda_i \lambda_j\x_{ik}^T \x_{jm}z_{ik}z_{jm}  -\sum_{i=1}^{n-1}\sum_{j=i+1}^{n} \sum_{m=1}^{|P_j|} \lambda_i \lambda_j\x_{jm}^T \x_{jm}z_{jm}.\end{align*} 

However, this remains a quadratic objective function due to the product $z_{ik}z_{jm}$ in the term
\[+2\sum_{i=1}^{n-1}\sum_{j=i+1}^{n} \sum_{k=1}^{|P_i|} \sum_{m=1}^{|P_j|}\lambda_i \lambda_j \x_{ik}^T \x_{jm}z_{ik}z_{jm}. \] 
To address this, we use some classical integer programming techniques. Recall that the $z_{ik}$ are $0,1$-variables. First, we introduce new variables $z_{ijkm}$ to represent each product $z_{ik} z_{jm}$. Second, we add linear constraints that link the new variables $z_{ijkm}$ to the original $z_{ik},z_{jm}$:
\begin{align*}
z_{ijkm} &\leq z_{ik} \\
z_{ijkm} &\leq z_{jm} \\
z_{ijkm} &\geq z_{ik}+z_{jm}-1. \end{align*}

The result is a linearization of the integer program in which $z_{ik}=0$ or $z_{jm}=0$ forces $z_{ijkm}=0$ due to the first two constraints, and $z_{ikjm}=1$ if $z_{ik}=z_{jm}=1$ due to the third constraint. This implies that $z_{ijkm}=z_{ik}z_{jm}$ for all $z_{ik}, z_{jm} \in \{0,1\}$. These constraints only need to be stated for pairs $i,j$ with $i<j$. In the following, the vector $\z=\binom{\z^1}{\z^2}$ lists all $\z^1=(z_{ik})$, followed by all $\z^2=(\z_{ijkm})$. We use the notation $\z\in \{0,1\}$ to denote that all its components lie in the specified domain.

We sum up the construction through a full formulation of the integer program for generating a combination of optimal reduced cost, and a corresponding formal statement. 

\begin{equation*}\label{Gen-IP}
\begin{array}{crll}
  \mathrm{max}  & \sum\limits_{i=1}^n \sum\limits_{k=1}^{|P_i|} y_{ik}z_{ik} - \sum\limits_{i=1}^{n-1}&\sum\limits_{j=i+1}^{n} \sum\limits_{k=1}^{|P_i|} \lambda_i \lambda_j \x_{ik}^T \x_{ik}z_{ik}-&\sum\limits_{i=1}^{n-1}\sum\limits_{j=i+1}^{n} \sum\limits_{m=1}^{|P_j|}\lambda_i \lambda_j \x_{jm}^T \x_{jm}z_{jm}\\ &+2\sum\limits_{i=1}^{n-1}\sum\limits_{j=i+1}^{n} \sum\limits_{k=1}^{|P_i|} \sum\limits_{m=1}^{|P_j|} & \lambda_i \lambda_j \x_{ik}^T \x_{jm}z_{ijkm}  \nonumber \tag{Gen-IP} \\
\mathrm{s.t.} & \sum\limits_{k = 1}^{|P_i|} z_{ik} & = 1  &\forall i = 1, \ldots, n \\
& z_{ijkm}  &\leq z_{ik}, &\forall i = 1, \ldots, n, \forall j = 1, \ldots, n, j > i,\\
&&&\forall k = 1, \ldots, |P_i|, \forall m = 1, \ldots, |P_j| \\
&z_{ijkm}  &\leq z_{jm}, &\forall i = 1, \ldots, n, \forall j = 1, \ldots, n, j > i,\\
&&&\forall k = 1, \ldots, |P_i|, \forall m = 1, \ldots, |P_j| \\
& z_{ijkm} &\geq z_{ik}+z_{jm}-1, &\forall i = 1, \ldots, n, \forall j = 1, \ldots, n, j > i,\\
&&&\forall k = 1, \ldots, |P_i|, \forall m = 1, \ldots, |P_j| \\
& \z  &\in \{0,1\} \nonumber 
\end{array}
\end{equation*}

\begin{theorem}\label{thm:main}
Given a set of possible combinations $S_0^*$, and an optimal fixed-support approximation of the barycenter problem over $S_0^*$, IP (\ref{Gen-IP}) finds a new combination of best reduced cost from the set of all combinations $S^*$.
\end{theorem}

Note that the input of $S_0^*$ and the best solution for that set of possible combinations are required for the specification of $y=(y_{ik})$. By contrast, the set of constraints of IP (\ref{Gen-IP}) is independent from this input. 


\subsection{A Simpler Mixed-Integer Program}\label{sec:mip}

Next, we simplify IP (\ref{Gen-IP}) to a smaller mixed-integer program (MIP), i.e., a program in which some of the variables are not required to be integer. The advantages over IP (\ref{Gen-IP}) will be twofold: in addition to the relaxation of some integer variables, we will be able to eliminate a type of main constraints.

To this end, we assume that $\x > 0$ in all measures. This assumption is not a restriction, as the region containing all $\x_{ik}$ can always be shifted to lie in the positive orthant of $\R^d$ without changing the optimal solution. This can be verified by recalling that $c_h = \sum_{i=1}^{n-1} \lambda_i \sum_{j=i+1}^{n} \lambda_j || \x_i^h - \x_j^h ||^2$ only takes into account (squared) pairwise distances of $\x_i^h,\x_j^h$. 
The consequence of this assumption is that $\x_{ik}^T \x_{jm}>0$ for all $\x_{ik},\x_{jm}$. 

We begin by showing that, in any optimal solution, $z_{ijkm}$ is chosen as large as possible, even without explicit specification of the constraints $z_{ijkm} \geq z_{ik}+z_{jm}-1$ or domain constraints $z_{ijkm}\in\{0,1\}$. 

\begin{lemma}\label{lem:auxvar}
Let $\x > 0$ and let $\z^*$ be an optimal solution to an LP relaxation of IP (\ref{Gen-IP}) without specification of the constraints  $z_{ikjm} \geq z_{ik}+z_{jm}-1$. Let further $\z^*_{ik} \in \{0,1\}$ for all $i = 1, \ldots, n$ and $k = 1, \ldots, |P_i|$. Then $\z^*\in \{0,1\}$.
\end{lemma}

\begin{proof}
We have to prove that the auxiliary variables $z^*_{ijkm}$ satisfy $z^*_{ijkm} \in \{0,1\}$. 
In the maximization objective, they only appear in the term
$$+2\sum\limits_{i=1}^{n-1}\sum\limits_{j=i+1}^{n} \sum\limits_{k=1}^{|P_i|} \sum\limits_{m=1}^{|P_j|}  \lambda_i \lambda_j \x_{ik}^T \x_{jm}z_{ijkm}.$$
Recall that $\x > 0$ by assumption, and that $\lambda > 0$ by definition. Thus, $\lambda_i \lambda_j \x_{ik}^T \x_{jm} \geq 0$. For fixed $z^*_{ik},z^*_{jm} \in \{0,1\}$, a maximal objective function value can only be achieved by choosing $z^*_{ijkm}$ as large as possible. As $z^*_{ijkm} \leq z^*_{ik},z^*_{jm}$, one obtains $z^*_{ijkm}= \min\{z^*_{ik},z^*_{jm}\} \in \{0,1\}$. This proves the claim.\hfill \qed\end{proof}

Lemma \ref{lem:auxvar} shows that, in an optimum of such a relaxation of IP (\ref{Gen-IP}), if $z^*_{ik}=z^*_{jm}=1$ then $z^*_{ijkm}=1$, and if $z^*_{ik}=0$ or $z^*_{jm}=0$ then $z^*_{ijkm}=0$. More generally, integrality of the $z^*_{ik}$ implies integrality of the auxiliary variables $z^*_{ijkm}$, and $z^*_{ijkm} \geq z^*_{ik}+z^*_{jm}-1$ is automatically satisfied. 
Thus, we can drop the constraints $z_{ijkm} \geq z_{ik}+z_{jm}-1$ and domain constraints $z_{ijkm}\in\{0,1\}$ from IP (\ref{Gen-IP}) to obtain an MIP that retains the same optimal set. Let us state this MIP; recall the notation $\z^1=(z_{ik})$.

\begin{equation*}\label{Gen-MIP}
\begin{array}{crll}
  \mathrm{max}  & \sum\limits_{i=1}^n \sum\limits_{k=1}^{|P_i|} y_{ik}z_{ik} - \sum\limits_{i=1}^{n-1}&\sum\limits_{j=i+1}^{n} \sum\limits_{k=1}^{|P_i|} \lambda_i \lambda_j \x_{ik}^T \x_{ik}z_{ik}-&\sum\limits_{i=1}^{n-1}\sum\limits_{j=i+1}^{n} \sum\limits_{m=1}^{|P_j|} \x_{jm}^T \x_{jm}z_{jm}\\ &+2\sum\limits_{i=1}^{n-1}\sum\limits_{j=i+1}^{n} \sum\limits_{k=1}^{|P_i|} \sum\limits_{m=1}^{|P_j|} & \lambda_i \lambda_j \x_{ik}^T \x_{jm}z_{ijkm}  \nonumber \tag{Gen-MIP} \\
\mathrm{s.t.} & \sum\limits_{k = 1}^{|P_i|} z_{ik} & = 1  &\forall i = 1, \ldots, n \\
& z_{ijkm}  &\leq z_{ik}, &\forall i = 1, \ldots, n, \forall j = 1, \ldots, n, j > i,\\
&&&\forall k = 1, \ldots, |P_i|, \forall m = 1, \ldots, |P_j| \\
&z_{ijkm}  &\leq z_{jm}, &\forall i = 1, \ldots, n, \forall j = 1, \ldots, n, j > i,\\
&&&\forall k = 1, \ldots, |P_i|, \forall m = 1, \ldots, |P_j| \\
& \z^1  &\in \{0,1\} \nonumber
\end{array}
\end{equation*}


In fact, the maximal choice of $z^*_{ijkm}= \min\{z^*_{ik},z^*_{jm}\}$ shown in the proof of Lemma \ref{lem:auxvar} also holds when $z^*$ is fractional, i.e., when $z^*_{ik},z^*_{jm} \neq 0,1$. This will be helpful in the further analysis and for the design of a viable solution approach in Section \ref{sec:IpvsLP}.



%


We sum up these observations with a formal statement on the equivalence of the optimal sets of IP (\ref{Gen-IP}) and MIP (\ref{Gen-MIP}).

\begin{theorem}\label{thm:main2}
Let $\x > 0$. Given a set of possible combinations $S_0^*$, and an optimal fixed-support approximation of the barycenter problem over $S_0^*$, MIP (\ref{Gen-MIP}) finds a new combination of best reduced cost from the set of all combinations $S^*$. The sets of optimal solutions $\z^*$ of IP (\ref{Gen-IP}) and MIP (\ref{Gen-MIP}) are identical.
\end{theorem}

We conclude this section with a closer look at the size of MIP (\ref{Gen-MIP}) in terms of the number of variables and constraints. For simplicity, let all $n$ measures have $p$ support points. Then there are $n\cdot p$ $0,1$-variables of type $z_{ik}$. In the following, we refer to support points by their variables $z_{ik}$. 

Each $z_{ik}$ is matched with $(n-1)\cdot p$ variables $z_{ijkm}$, one for each point $z_{jm}$ in another measure. As, $i<j$, there exist $\frac{1}{2}((n\cdot p) \cdot ((n-1) \cdot p))= \binom{n}{2}\cdot p^2$ continuous variables $z_{ijkm}$. Together, there are a total of $$n\cdot p \;\;\text{$0,1$-variables and } \binom{n}{2}\cdot p^2\;\; \text{continuous variables.}$$

The main constraints are comprised of two types: first, there are $n$ equalities $\sum_{k = 1}^{p} z_{ik} = 1$, one for each measure $P_i$. Second, for each $z_{ijkm}$, there are two inequalities $z_{ijkm}\leq z_{ik}, z_{jm}$. As $i<j$, this gives $\binom{n}{2} \cdot 2p^2 = n(n-1)\cdot p^2$ constraints. 
In total, we have $$n+n(n-1)\cdot p^2 \;\;\text{main constraints}.$$

Thus, MIP (\ref{Gen-MIP}) has a quadratic number of continuous variables and main constraints in the number $n$ of measures and number $p$ of support points in each measure, as well as a linear number of $0,1$-variables. While this contrasts sharply with the number $p^n$ of possible combinations, and while it is a significant improvement over IP (\ref{Gen-IP}) (where the $z_{ijkm}$ are $0,1$-variables and each has an additional main constraint), of course, MIPs with such scaling generally quickly become hard to solve. In the next section, we discuss our strategy for and implementation of a solution approach.

\section{Solution Approach and Practical Implementation}\label{sec:IpvsLP}

We begin with the design of a viable solution approach to MIP (\ref{Gen-MIP}), informed by the observations in Section \ref{sec:mip}. Then, we present some computational results and discuss advantages and disadvantages of our approach. 

\subsection{A Branch-and-Bound Strategy}\label{sec:bandb}

Our main goal is the design of a strategy for the solution of MIP (\ref{Gen-MIP}). In general, mixed-integer programs can be difficult to solve. However, the observations in Section \ref{sec:mip} lead to a quite natural and promising approach. Recall that the variables that are required to be integer are only the $z_{ik}$: if and only if these form a $0,1$-vector, the result corresponds to a combination of support points from each measure.

By Theorem \ref{thm:main2}, the optimal sets of MIP (\ref{Gen-MIP}) and IP (\ref{Gen-IP}) are identical. This allows an immediate application of Lemma \ref{lem:auxvar} to an LP relaxation of MIP (\ref{Gen-MIP}): if an optimal solution $\z^*$ in such a relaxation exhibits $z^*_{ik}\in \{0,1\}$, then all variables $z^*_{ijkm}$ `automatically' satisfy $z^*_{ijkm}\in \{0,1\}$, too. This means that a branch-and-bound approach for solving MIP (\ref{Gen-MIP}) requires branching only over the set of $n\cdot p$ $0,1$-variables $z_{ik}$. We are going to use such a branch-and-bound approach.

First, let us take a closer look at the LP relaxation of MIP (\ref{Gen-MIP}). Typically, constraints of type $\z^1 \in \{0,1\}$ are relaxed to $0\leq \z^1 \leq 1$. However, note that $\z^1\leq 1$ is implied by $\sum_{k = 1}^{|P_i|} z_{ik}  = 1 $ for all $i = 1, \ldots, n$. Thus it suffices to add the constraints $0\leq \z^1$ to the system. We arrive at the following LP.

\begin{equation*}\label{Gen-LP}
\begin{array}{crll}
  \mathrm{max}  & \sum\limits_{i=1}^n \sum\limits_{k=1}^{|P_i|} y_{ik}z_{ik} - \sum\limits_{i=1}^{n-1}&\sum\limits_{j=i+1}^{n} \sum\limits_{k=1}^{|P_i|} \lambda_i \lambda_j \x_{ik}^T \x_{ik}z_{ik}-&\sum\limits_{i=1}^{n-1}\sum\limits_{j=i+1}^{n} \sum\limits_{m=1}^{|P_j|} \x_{jm}^T \x_{jm}z_{jm}\\ &+2\sum\limits_{i=1}^{n-1}\sum\limits_{j=i+1}^{n} \sum\limits_{k=1}^{|P_i|} \sum\limits_{m=1}^{|P_j|} & \lambda_i \lambda_j \x_{ik}^T \x_{jm}z_{ijkm}  \nonumber \tag{Gen-LP} \\
\mathrm{s.t.} & \sum\limits_{k = 1}^{|P_i|} z_{ik} & = 1  &\forall i = 1, \ldots, n \\
& z_{ijkm}  &\leq z_{ik}, &\forall i = 1, \ldots, n, \forall j = 1, \ldots, n, j > i,\\
&&&\forall k = 1, \ldots, |P_i|, \forall m = 1, \ldots, |P_j| \\
&z_{ijkm}  &\leq z_{jm}, &\forall i = 1, \ldots, n, \forall j = 1, \ldots, n, j > i,\\
&&&\forall k = 1, \ldots, |P_i|, \forall m = 1, \ldots, |P_j| \\
& \z^1  &\geq 0 \nonumber
\end{array}
\end{equation*}
For a simple discussion, we again assume that all $n$ measures have $p$ support points for the remainder of this section. LP (\ref{Gen-LP}) then has the same number of $n\cdot p + \binom{n}{2}\cdot p^2$ variables as the MIP, and an additional $n\cdot p$ constraints that replace the $0,1$-domain of the $z_{ik}$, for a total of $n+n(n-1)\cdot p^2+np$ constraints. This quadratic scaling leads to relatively low memory requirements for moderate problem sizes; see Section \ref{sec:computations}.

It suffices to search for an {\em optimal, integral vertex} $\z^*$ of LP (\ref{Gen-LP}) to match the optimum for MIP (\ref{Gen-MIP}) due to containment of the optimal set of the LP in the $[0,1]$-unit hypercube in the underlying space. To measure `how far' one is from having an integer solution, we now study the number of fractional components -- we refer to this as {\em fractionality} -- of a vertex of LP (\ref{Gen-LP}). Early computations immediately revealed that LP (\ref{Gen-LP}) can lead to fractional optimal vertices. In particular, while the constraint matrix has only $0,1$-entries, it is not totally-unimodular; we provide a brief argument in Appendix \ref{app:TU}. 

The study of fractionality serves two purposes. First, we show that there are not only fully-fractional solutions, i.e., solutions where all components are fractional, but even fully-fractional {\em vertices}. This further justifies the use of classical integer programming techniques, such as branch-and-bound, to `break up' a high fractionality. In other situations, specifically if the number of fractional variables is provably low, simpler approaches, such as tailored rounding routines, sometimes suffice in practice. Such an example appeared in our previous work, for example, in \cite{bbg-11}. Second, we are able to identify a repeated structure in the fractional components of a vertex. This helps with the selection of promising branching rules.


We begin by showing the existence of a fully-fractional vertex of LP (\ref{Gen-LP}). It is not hard to state a fully-fractional solution $\z^*$ of LP (\ref{Gen-LP}). Consider an input of $n\geq 2$ measures of $p\geq 2$ support points. Then $\z^*$ defined by setting $z^*_{ik}=\frac{1}{p}$ for all $i\leq n, k\leq p$ and $z^*_{ijkm}=\min\{z^*_{ik}, z^*_{jm}\}=\frac{1}{p}$ clearly gives a feasible solution to LP (\ref{Gen-LP}). We now show that $z^*$ is, in fact, a vertex.

\begin{lemma}\label{lem:fullyfractional}
Let $n\geq 2, p\geq 2$. Then $\z^*$ defined by $z^*_{ik}=\frac{1}{p}$ for all $i\leq n, k\leq p$ and $z^*_{ijkm}=\frac{1}{p}$ for all $i\leq n, j\leq n, j> i$ and $ k\leq p, m \leq p$ is a vertex of LP (\ref{Gen-LP}). 
\end{lemma}

\begin{proof}
Recall that, given a representation of a polyhedron, to form a vertex there has to exist a set of active constraints whose normals form a matrix of full rank equal to the dimension. In other words, one have to identify a row submatrix of active constraints with a rank equal to the dimension of the underlying space, i.e., equal to the number of variables $n\cdot p + \binom{n}{2}\cdot p^2$ of LP (\ref{Gen-LP}).

To this end, note that the second and third types of constraints of LP (\ref{Gen-LP}), rewritten as 
\begin{equation*}\label{LP'}
\begin{array}{crlll}
& z_{ijkm} - z_{ik} \quad &\leq 0, & \;\;\; \forall i = 1, \ldots, n, \forall k = 1, \ldots, p \\
&z_{ijkm} - z_{jm} \quad &\leq 0, & \;\;\; \forall j= 1, \ldots, n, j>i, \forall m = 1, \ldots, p
\end{array}
\end{equation*}
form a collection of rows with a single $+1$ and a single $-1$, and $0$ everywhere else. The corresponding row submatrix is the transpose of the node-arc incidence matrix of a directed, bipartite graph $G=(V,E)$ with vertex set $V=A\dot{\cup} B$, where $A$ is the set of $z_{ijkm}$, $B$ the set of $z_{ik}$, and $E$ the set of directed arcs from $A$ to $B$ that connect each $z_{ijkm}$ with $z_{ik}$ and $z_{jm}$.

For a node-incidence matrix, a column set is linearly independent if and only if it does not form a cycle in the underlying undirected graph. Based on this idea, we identify a set $T$ of $n\cdot p + \binom{n}{2}\cdot p^2 -1$ arcs in this graph that do not form a cycle. These arcs correspond to a row submatrix of rank $n\cdot p + \binom{n}{2}\cdot p^2 -1$. Afterwards, we then have to identify only one additional, independent row from LP (\ref{Gen-LP}).

First, note that each $z_{ijkm}$ is connected to only $z_{ik}$ and $z_{jm}$. To form a cycle in $G$ that includes a $z_{ijkm}$, one has to select {\em both} incident arcs. This allows us to begin construction of $T$ by adding {\em exactly one} of these two arcs for all $z_{ijkm}$, which gives $\binom{n}{2}\cdot p^2$ arcs. Informally, we start with a maximal matching in the underlying bipartite graph. 

It remains to identify $n\cdot p -1$ further arcs that can be added to $T$ without closing a cycle. First, we add all the missing $(n-1)\cdot p$ arcs to connect $z_{11}$ to all $z_{jm}$ for $2\leq j\leq n$, $m\leq p$ by a path of two edges incident to $z_{1j1m}$. Second, we add all the missing $p-1$ arcs to connect $z_{21}$ to all $z_{1m}$ for $2\leq m\leq p$ by a path of two edges incident to $z_{12m1}$.

We arrive at a total of $(n-1)\cdot p + (p-1)= n\cdot p -1$ additional arcs that do not form a cycle in the graph: only $z_{11}$ and $z_{21}$ are connected (by paths of two edges) to other nodes of type $z_{ik}$, but a cycle would have to include at least one more $z_{ik}$. 
Informally, we have constructed a spanning tree in the underlying bipartite graph.

It remains to identify one additional independent row in the system. Of course, the equality constraints $\sum_{k = 1}^{p} z_{ik}  = 1$ for $i\leq n$ are always satisfied. The rows corresponding to set $T$ (augmented with their right-hand sides $0$) form a system with feasible solution set $z=c\cdot (1,\dots,1)^T$ for all $c \in \mathbb{R}$, and adding any of the equalities $\sum_{k = 1}^{p} z_{ik}  = 1$ gives precisely $c=\frac{1}{p}$. Thus, we add the row of one of the equalities ($\sum_{k = 1}^{p} z_{ik}  = 1$ for any $i$) to $T$ and obtain a row submatrix of full rank. This shows that $\z^*$ is a vertex of LP (\ref{Gen-LP}). \qed
\end{proof}

While the above example might seem pathological, highly and fully-fractional solutions do appear in our practical computations in Section \ref{sec:computations}; c.f., Table \ref{table:frac}. As we will see, classical branching rules, such as branching on lowest-index variables or those closest or furthest from integrality, all are viable to break up the high fractionality. 

We conclude this section by proving a sufficient property for an optimal vertex $\z^*$ of LP (\ref{Gen-LP}) to be integral. We show that the fractional values in $\z^1$, i.e., the variables $z^*_{ik},z^*_{jm}$, cannot all be different. In In fact, there must be a repeated value $z^*_{ik}=z^*_{jm}$ for some $i\neq j$ and some $k,m$. Otherwise $\z^*\in\{0,1\}$. 

\begin{lemma}\label{thm:auxvar}
Let $n\geq 2, p\geq 2$. Let further $\z^*$ be an optimal vertex of LP (\ref{Gen-LP}) and let $z^*_{ik} \neq z^*_{jm}$ for all $0<z^*_{ik},z^*_{jm}<1$ with $i \neq j$. Then $\z^*\in \{0,1\}$.
\end{lemma}

\begin{proof}
We prove the claim by contradiction. Assume $\z^* \notin \{0,1\}$ and $z^*_{ik}\neq z^*_{jm}$ for all $0<z^*_{ik},z^*_{jm}<1$ with $i \neq j$. We will show that the number of (independent) active constraints is strictly less than the number $n\cdot p + \binom{n}{2}\cdot p^2$ required for a vertex of LP (\ref{Gen-LP}). We now build this set $T$ of active constraints.

Note that the $n$ equality constraints are always active and independent from each other; they are added to $T$. Note that variables $z_{ik}$ only appear in the $i$-th equality constraints and that, for each $i$, at least one $z^*_{ik}$ satisfies $z^*_{ik}>0$. Thus one can add any active domain constraints $z_{ik}\geq 0$ to $T$ while retaining linear independence. As $\z^* \notin \{0,1\}$, at most $n\cdot (p-1) -1$ domain constraints of type $z_{ik}\geq 0$ are active. At this point, the size $|T|$ of $T$ is bounded above by $|T| \leq n + n\cdot (p-1) -1= n\cdot p -1 $.

The remaining $\binom{n}{2}\cdot p^2 +1$ active, independent constraints need to come from the auxiliary constraints
 \begin{align*}
z_{ijkm} & \leq z_{ik} \\
z_{ijkm} & \leq z_{jm}
\end{align*}
for the different $z_{ijkm}$. As $\z^*$ is optimal, $z^*_{ijkm}=\min\{z^*_{ik}, z^*_{jm}\}$. Thus at least one of the auxiliary constraints is active for each $z_{ijkm}$. Such a constraint can be added to $T$ while retaining linear independence, as the variable $z_{ijkm}$ is not used in any of the other constraints in $T$.

As $z^*_{ik} \neq z^*_{jm}$ for all $0<z^*_{ik},z^*_{jm}<1$ with $i \neq j$, only one auxiliary constraint is active whenever at least one of the variables is fractional. That constraint is added to $T$. For $z^*_{ik} = z^*_{jm} = 0$ or $=1$, however, both constraints are active. If $z^*_{ik} = z^*_{jm} = 0$, the corresponding domain constraints $z_{ik}\geq 0$ and $z_{jm}\geq 0$ already are in $T$. The system
 \begin{align*}
z_{ijkm} & \leq z_{ik} \\
z_{ijkm} & \leq z_{jm}  \\
z_{ik} & \geq 0\\
z_{jm} & \geq 0
\end{align*}
only has rank $3$, as it involves $3$ variables. Thus the second constraint involving $z_{ijkm}$ may not be added to $T$ without violating independence.

Finally, if $z^*_{ik} = z^*_{jm} = 1$, then the system
 \begin{align*}
 \sum\limits_{l = 1}^{p} z_{il} & = 1\\
z_{il} & \geq 0 \;\; \forall l\leq p, l\neq k
 \end{align*}
is of full rank $p$. The same holds for 
 \begin{align*}
 \sum\limits_{l = 1}^{p} z_{jl} & = 1\\
z_{jl} & \geq 0 \;\; \forall l\leq p, l\neq m.
 \end{align*}
Combining these two systems with the constraints
  \begin{align*}
z_{ijkm} & \leq z_{ik} \\
z_{ijkm} & \leq z_{jm}
\end{align*}
gives a system of rank $2p+1$, as it involves $2p+1$ variables. Again, the second constraint involving $z_{ijkm}$ may not be added to $T$ without violating independence.

 Thus exactly $\binom{n}{2}\cdot p^2$ auxiliary constraints -- one for each $z_{ijkm}$ -- can be added to $T$ while retaining linear independence. The total number of active constraints is  at most $n\cdot p +\binom{n}{2}\cdot p^2 -1<n\cdot p + \binom{n}{2}\cdot p^2$, strictly less than what is required for a vertex of LP (\ref{Gen-LP}). This proves the claim. \qed
\end{proof}

Lemma \ref{thm:auxvar} implies that there must be repeats of a fractional value in any fractional optimal vertex. This observation makes it promising to try a branching rule that prioritizes the elimination of such repeats and compare its performance to standard branching rules. Further, such a repeat has to appear for fractional values in {\em different} measures. 
Thus at least two measures are split fractionally; at least four variables are fractional. This means that the depth of a branch-and-bound tree is bounded by $n\cdot p-3$, regardless of the branching rule.

\subsection{Computational Experiments}\label{sec:computations}
The purpose of our computational experiments is twofold: one, to confirm the predicted behavior of solving LP (\ref{Gen-LP}) in regards to the fractionality of the vertices, and two, to explore the practicality of using MIP (\ref{Gen-MIP}). To these ends, we implemented a column generation routine using the Gurobi Mixed Integer Program (MIP) solver using the C++ API; source code is available at \url{https://github.com/StephanPatterson/PricingIP}. This MIP solver, as is standard in commercial solvers, is based on branch-and-bound, runs in parallel, and by default supplements with cutting plane methods and other heuristics. We also implemented a custom branch-and-bound routine fully in C++ (that only uses the Gurobi Optimizer as part of its node exploration) for comparing different branching strategies; this source code is available at \url{https://github.com/StephanPatterson/IP-Pricing-BB}.

When solving MIP (\ref{Gen-MIP}) through the MIP solver in Gurobi, we found the solution times to be highly influenced by solver options. We first disabled the cutting plane methods, as adding constraints (particularly those without the same structure currently present in the problem) can increase solution times. Disabling heuristics also improved solution times, while disabling the presolver dramatically increased solution times, and remains enabled in the following experiments. We also set the MIP solver to focus on proving optimality (MIPFocus = 2).

Our experiments were run on a laptop (MacBook Pro, 2.4 GHz Intel Core i9, 32 GB of RAM, SSD) with 16 cores for parallel computations. The experiments consist of running column generation on $n$ measures, with $n$ between $3$ and $36$, and each measure containing $2-12$ support points. Note that these problem sizes, especially for the larger values of $n$ are already very challenging for an exact barycenter computation \cite{abm-16,bp-18}. We report on the averages for many repeated runs, with $10-100$ repeats depending on problem size. The data for the experiments consists of event locations given in latitude and longitude (support points). These locations do not underlie an obvious structure, which is the hardest setting for barycenter computations \cite{ab-21,bp-18}. The events are grouped by the month and year in which they occurred (measures).

We compare the use of a previous column generation strategy, based on a direction evaluation of the reduced cost vector $y^T A - c$ as in \cite{bp-21}, with the new approaches based on MIP (\ref{Gen-MIP}). For comparisons involving memory requirements or computation times, we relate to the direct use of the Gurobi MIP solver for MIP (\ref{Gen-MIP}). For information on fractionality, best branching strategies, and the size of the branch-and-bound tree, we relate to our custom implementation. Throughout this section, we refer to the direct evaluation of the reduced cost vector as the {\em classic} approach and to the ones based on MIP (\ref{Gen-MIP}) as the {\em MIP} approach. 

\subsubsection{A Trade-off of Memory and Setup Times.}

We begin with a memory comparison. As expected -- it was our main goal -- the use of a mixed-integer program provides significant savings in required memory over the classic approach. Table \ref{tab:mem} exhibits the gap measured after the first solve of the pricing problem and the addition of the first column to the master problem. In the tables in this section, $n$ refers to the number of measures, {\em support} refers to the total number of support points, and {\em variables} refers to the number of variables in MIP (\ref{Gen-MIP}). As problem size grows, the memory gap widens and reaches more than two orders of magnitude.

\begin{table}[]
    \centering
    {\begin{tabular}{|c|c|c|c|c|} \hline
   \; n \; & \; Support \; & \; Variables \; & \; Classic \; & \; MIP \; \\ \hline
    12 & 42 & 1,990,656 & 53 & 75 \\ \hline
    14 & 57 & 28,449,792 & 661 & 104 \\ \hline
    15 & 50 & 31,850,496 & 607 & 94 \\ \hline
    18 & 56 & 254,803,968 & 5,730 & 110 \\ \hline
    15 & 71 & 731,566,080 & 10,910 & 143 \\ \hline
    16 & 60 & 1,019,215,872 & 22,910 & 146 \\ \hline
    17 & 73 & 1,567,641,600 & 23,370 & 139 \\ \hline 
    16 & 69 & 1,820,786,688 & 41,000 & 146 \\ \hline
    \end{tabular}\vspace*{0.25cm}}
    \caption{Total memory used in MB.}
    \label{tab:mem}
\end{table}

By contrast, even with the above described solver settings, the MIP approach only shows improvement in solver times for the largest experiments from our data set. In Table \ref{tab:mem2}, we exhibit the gap in average time per solution of the pricing problem. Especially for small problem sizes, our experiments reveal a noticeable cost for the setup of the MIP that is not needed for the classic approach; see the first row in the table. For the largest problems, however, the MIP approach is competitive.  


\begin{table}[]
    \centering
    {\begin{tabular}{|c|c|c|c|c|} \hline
    \; n \; & \; Support \; & \; Variables \; & \; Classic \; & \; MIP \; \\ \hline
    12 & 42 & 1,990,656 & 0.068& 4.18 \\ \hline
    15 & 50 & 31,850,496 & 1.12 & 10.07 \\ \hline
    14 & 57 & 28,449,792 & 1.16 &18.450  \\ \hline
    18 & 56 & 254,803,968 & 9.09 &  18.12 \\ \hline
    15 & 71 & 731,566,080 & 93.86 & 110.39 \\ \hline
    19 & 60 & 1,019,215,872 & 30.73 & 22.30 \\ \hline
    17 & 73 & 1,567,641,600 & 90.04 & 115.28 \\ \hline
    16 & 69 & 1,820,786,688 & 126.47 & 70.22 \\ \hline
    \end{tabular}\vspace*{0.25cm}}
    \caption{Average time per solution of the pricing problem in seconds.}
    \label{tab:mem2}
\end{table}

Summing up, the MIP approach performs as expected and intended. It cannot resolve the underlying challenge in the computations, but successfully performs a trade-off of longer setup times for lower memory requirements. 
This facilitates a scaling to larger problem sizes. In practice, the choice of which method to use may depend largely on the available RAM memory: the classic approach outperforms the MIP approach as long as the exponentially-scaling input for the direct reduced cost computation fits in RAM. When the classic computations would have to make use of virtual RAM located on a hard drive, they slow down dramatically, and it becomes better to use our proposed MIP approach. 

\subsubsection{Fractionality of Solutions.}

Next, we turn to the fractionality of solutions in LP (\ref{Gen-LP}), the relaxation of MIP (\ref{Gen-MIP}). In Lemma \ref{lem:fullyfractional}, we saw there exist pathological vertex solutions that are fully fractional. 

To measure the fractionality in practice, we generated random subsets of our input measures and an initial set of combinations for a feasible solution, in turn giving us the vector $y$ for the pricing objective. Then, we solved the resulting LP (\ref{Gen-LP}) and measured the percentage of the variables $z_{ik}$ which were fractional, along with the number of unique (different) fractional values.

Table \ref{table:frac} shows the results of these computations. They align with the theoretical analysis in Section \ref{sec:bandb} and further justify the design and use of a branch-and-bound method. The number of fractional values is very high throughout, sometimes reaching $100\%$, while the number of unique values is low. A good branching strategy would be one that not only `attacks' the high fractionality (in a sense, any branching strategy would do that), but prioritizes a reduction in the number of unique fractional values, which are already low to begin with.

\begin{table}[]
    \centering
    {\begin{tabular}{|c|c|c|c|} \hline
    \; n \; & \; Support \; & \; Fractional Values \; & \; Unique Fractional Values \; \\ \hline
    3 & 21 & 100\% & 1 \\ \hline
    8 & 29 & 86.2\% & 4 \\ \hline
    9 & 41 & 87.8\% & 5 \\ \hline
    12 & 57 & 100\% & 5 \\ \hline
    14 & 63 & 98.4\% & 6 \\ \hline
    18 & 56 & 96.4\% & 3 \\ \hline
    23 & 71 & 70.4\% & 3 \\ \hline
    36 & 132 & 84.8\% & 5 \\ \hline
    \end{tabular}\vspace*{0.25cm}}
    \caption{Percentage of variables $z_{ik}$ observed to be fractional when solving LP (\ref{Gen-LP}) and the number of unique fractional values.}\label{table:frac}
\end{table}

\subsubsection{Design of Branch-and-Bound for MIP (\ref{Gen-MIP}).} Finally, we discuss the design of a custom branch-and-bound method to solve MIP (\ref{Gen-MIP}). 
We begin with some general information and then turn to a comparison of different branching strategies. The method begins with solving the relaxation LP (\ref{Gen-LP}), whose solution provides an upper bound (or `maximum potential') for the objective value of MIP (\ref{Gen-MIP}). We will explore different strategies for choosing a variable $z_{ik}$ for branching; once chosen, in the left-hand subtree, $z_{ik}$ is set to $0$ and in the right-hand subtree, it is set to $1$. Since at most one decision per variable is required to reach integrality, the maximum height of the produced tree is $\sum_{i=1}^n |P_i|$; in fact, the more refined analysis at the end of Section \ref{sec:bandb} exhibits that this bound can be improved to $t_\text{B\&B}:=(\sum_{i=1}^n |P_i|) - 3$. This results in a theoretical maximum number of nodes in the tree up to $2^{t_\text{B\&B}}-1$. 
As we will see, in our practical computations, we only visit an extremely small fraction of these exponentially many possible nodes. One of the reasons for this is that by setting a variable to $1$, one implicitly sets all other variables for the same measure to $0$. The depth of the tree where pruning occurs, and how many nodes are avoided, will depend on the branching strategy, but, regardless, many right-hand branches will be pruned during the process.  
Further, our branch-and-bound implementation begins with an initial lower bound that also immediately allows for pruning branches that fall below it.  

Recall that the initialization of column generation requires an initial feasible solution.  We generate such an initial solution using the greedy strategy described in \cite{bp-21}, which iteratively creates combinations of support points out of the first support points in each measure available to receive mass, and assigns to each generated combination the maximum mass receivable by those support points. The generated solution has the aforementioned second use: the total transportation cost associated with it leads to an initial lower bound on the optimal value of MIP (\ref{Gen-MIP}) for the purpose of pruning nodes.


During the custom branch-and-bound, processing a node creates exactly two child nodes and the resulting (relaxed) MIPs are solved immediately. Solving the two new MIPs is implemented using the new Multiple Scenario feature in Gurobi 9.0. The assigned values of the branching variable (left-hand branch $0$, right-hand branch $1$) is achieved using the scenario upper bound and lower bound attributes, ScenNUB and ScenNLB, by setting the left-hand branch to have an upper bound of $0$, and the right-hand branch to a lower bound of $1$. Gurobi's output reports that the two scenarios are infeasible only if both scenarios are infeasible, so we check each scenario individually for infeasibility and prune the tree accordingly.

We are now ready to explore three branching strategies. Our goal is to identify a strategy that produces the smallest branch-and-bound tree, i.e., that processes the fewest nodes.

First, we consider a naive direct branching strategy in which the variables are branched on in the order they are given. That is, first $z_{11}$ is processed, then $z_{12}$, until the last support point of measure $P_1$ is reached; then the branching moves on to the points in measure $P_2$, and so on. We refer to this simple strategy as {\em Index Order}.

Second, we consider a {\em Closest to Integer} strategy in which the variable whose value is closest to $0$ or $1$ (but still some small tolerance away from integer to avoid issues of numerical stability) is branched on. Essentially this is a rounding strategy, based on the general idea that, in practice, a good solution to an integer program may sometimes be found by rounding a relaxed solution to integer values.

Third, we explore a strategy in which we select a (first) variable whose current fractional value is repeated most frequently in all variables. As observed earlier, fractional solutions in which all variables are assigned the same value are both theoretically possible and appear in practice. The goal of this branching strategy is to break apart this highly repetitious fractionality the most directly. We refer to this strategy as {\em Most Repeated}.

\begin{table}[]
    \centering
    {\begin{tabular}{|c|c|c|c|c|} \hline
        \; n \; & \; Support \; & \multicolumn{3}{c|}{Nodes Processed: Unsorted}  \\ \hline
         && \; Index Order \; & \; Closest to Integer \; & \; Most Repeated \; \\ \hline
         5 & 32 & 2282 & 2281 & 1028 \\ \hline
         6 & 34 & 4567 & 3755 & 2057 \\ \hline
         7 & 31 & 2870 & 3007 &  2204 \\ \hline
         8 & 29 & 2247 & 2032 & 1727 \\ \hline
         9 & 26 & 4225 & 3071 & 3071 \\ \hline
         9 & 32 & 15423 & 9366 & 8639 \\ \hline
         9 & 34 & 7212 & 4055 & 3455 \\ \hline
         10 & 29 & 15489 & 6143 & 6143 \\ \hline
    \end{tabular}\vspace*{0.25cm}}
    \caption{The number of nodes processed for three branching strategies. The measures were passed to the algorithm in their naturally occurring order.} 
    \label{tab:unsorted_bb}
\end{table}

In our initial experiments, we leave the measures in their naturally occurring order (`unsorted'). Results are shown in Table \ref{tab:unsorted_bb}. We observed that the fewest nodes are processed for the {\em Most Repeated} strategy, in which we focus on breaking apart the highly fractional nature of the solutions of the relaxation. This aligns with the results of our theoretical analysis in Section \ref{sec:bandb}, where we identified such a strategy as a particularly promising angle of attack.

The naive {\em Index Order} strategy performed worst. The performance of the rounding strategy {\em Closest to Integer} was generally between that of the {\em Index Order} and {\em Most Repeated} strategies: for larger instances it performed almost as well as the {\em Most Repeated} strategy; for smaller instances it performed similarly to the {\em Index Order} strategy. 
It is noteworthy that in all strategies, only an extremely small percentage of the theoretical maximum number of nodes of a branch-and-bound tree on the corresponding number of variables is ever explored (for these experiments, $2^{t_\text{B\&B}}$ lies between $10^7$ and $10^8$). 

In previous work on column generation \cite{bp-21}, we observed that sorting the measures so that the smallest measures (lowest number of support points) are first can be beneficial for solution times. We theorized that such a sorting may also be beneficial in the our branch-and-bound method, as particularly in the {\em Index Order} strategy, the variables for entire measures would be fixed at high levels of the tree. Table \ref{tab:sorted_bb} shows how the number of processed nodes changes with this preprocessing. The number of processed nodes dropped significantly for all strategies. As expected, the benefit is the largest for the {\em Index Order} strategy, but  {\em Most Repeated}, which focuses on breaking apart the highly repetitious fractionality of solutions, still outperforms the other strategies. 

\begin{table}[]
    \centering
    {\begin{tabular}{|c|c|c|c|c|} \hline
        \; n \; & \; Total Support \; & \multicolumn{3}{c|}{Nodes Processed: Sorted}  \\ \hline
         && \; Index Order \; & \; Closest to Integer \; & \; Most Repeated \; \\ \hline
         5 & 32 & 1560 & 1785 & 1028 \\ \hline
         6 & 34 & 3123 & 2941 & 2057 \\ \hline
         7 & 31 & 2185 & 2748 & 1575 \\ \hline
         8 & 29 & 1564 &1339 & 1080 \\ \hline
         9 & 26 & 4696 & 4292 & 3240 \\ \hline
         9 & 32 & 1597 & 1027 & 1023  \\ \hline
         9 & 34 & 4426 & 3768 & 3456 \\ \hline
         10 & 29 & 4797 & 3238 & 3071 \\ \hline
    \end{tabular}\vspace*{0.25cm}}
    \caption{The number of nodes processed for three branching strategies. The measures were presorted so that those with the fewest support points were first.}
    \label{tab:sorted_bb}
\end{table}

\section{Conclusion and Outlook}\label{sec:conclusion}

Fixed-support approximations are the arguably most popular approach to the barycenter problem, in particular for exact barycenter computations. In this paper, we devised an integer program, IP (\ref{Gen-IP}), to generate an optimal support point to add to such a fixed support for a better approximation. We proved that a simpler MIP (\ref{Gen-MIP}) can be solved instead of IP (\ref{Gen-IP}), and we designed a branch-and-bound method that uses the symmetry of the underlying problem for a tailored branching strategy. The main advantage of the new approach is a significant drop in memory requirement. While it comes at a trade-off of longer setup times, it facilitates larger-scale computations than before. 
We see two practical uses of the methods in this paper. First, one can perform a limited number of iterations to improve on a given approximate barycenter. 
Second, the ability to find exact barycenters for larger instances than before leads to a more reliable assessment of the quality of solutions returned by algorithms in the literature. 

There are a few promising directions for practical and theoretical future work. Commercial solvers such as Gurobi have branch-and-bound routines that are highly optimized, for example allowing for the parallel solution of several subproblems, but do not allow the integration of a tailored branching strategy, as we propose. Our custom implementation is a proof of concept designed to identify a best branching strategy and to identify the impact of presorting measures by size. It would not be hard to improve the custom implementation and close some of the performance gap. We are also planning to compare different options of commercial solvers, and follow their development cycle -- additional options for tweaking algorithms are added regularly -- for a better transfer of our strategies to a state-of-the-art solver.

We see the most exciting research direction in striving for a deeper understanding of the set of all possible combinations $S^*$, whose size and scaling is the bottleneck in exact barycenter computations and also for the column generation approach in this work and previous work. The size of $S^*$ drives the exponential scaling of the reduced-cost vector in \cite{bp-18} and determines the size of IP (\ref{Gen-IP}). In all of these settings, the ability to a priori rule out combinations from $S^*$ that cannot appear in an optimal solution may dramatically improve practical performance. For other combinations, it may be possible to bound the approximation error incurred if left out of consideration. This would be a step to leaving the setting of exact barycenter computations and working towards the efficient computation of improving support points for approximate large-scale computations.




\section*{Acknowledgments}
The first author was supported by Air Force Office of Scientific Research grant FA9550-21-1-0233 and NSF grant 2006183, Algorithmic Foundations, Division of Computing and Communication Foundations.

\bibliography{barycenters_literature}

\begin{thebibliography}{10}

\bibitem{ac-11}
M.~Agueh and G.~Carlier.
\newblock Barycenters in the {W}asserstein space.
\newblock {\em SIAM Journal on Mathematical Analysis}, 43(2):904--924, 2011.

\bibitem{ab-21b}
J.~Altschuler and E.~Boix-Adser{\`a}.
\newblock Wasserstein barycenters can be computed in polynomial time in fixed
  dimension.
\newblock {\em Journal of Machine Learning Research}, 22:1--19, 2021.

\bibitem{ab-21}
J.~Altschuler and E.~Boix-Adser{\`a}.
\newblock {Wasserstein barycenters are NP-hard to compute}.
\newblock {\em SIAM Journal on Mathematics of Data Science}, 4(1):179--203,
  2022.

\bibitem{abm-16}
E.~Anderes, S.~Borgwardt, and J.~Miller.
\newblock Discrete {W}asserstein {B}arycenters: {O}ptimal {T}ransport for
  {D}iscrete {D}ata.
\newblock {\em Mathematical Methods of Operations Research}, 84(2):389--409,
  2016.

\bibitem{bhp-13}
M.~Beiglb\"ock, P.~Henry-Labordere, and F.~Penkner.
\newblock Model-independent bounds for option prices -- a mass transport
  approach.
\newblock {\em Finance and Stochastics}, 17(3):477--501, 2013.

\bibitem{bccnp-14}
J.-D. {Benamou}, G.~{Carlier}, M.~{Cuturi}, L.~{Nenna}, and G.~{Peyr{\'e}}.
\newblock {Iterative Bregman Projections for Regularized Transportation
  Problems}.
\newblock {\em SIAM Journal on Scientific Computing}, 37(2):A1111--A1138, 2015.

\bibitem{b-17}
S.~Borgwardt.
\newblock {An LP-based, Strongly Polynomial 2-Approximation Algorithm for
  Sparse Wasserstein Barycenters}.
\newblock {\em Operational Research}, 22:1511--1551, 2022.

\bibitem{bbg-11}
S.~Borgwardt, A.~Brieden, and P.~Gritzmann.
\newblock Constrained minimum-$k$-star clustering and its application to the
  consolidation of farmland.
\newblock {\em Operational Research}, 11(1):1--17, 2011.

\bibitem{bp-18}
S.~Borgwardt and S.~Patterson.
\newblock Improved {L}inear {P}rograms for {D}iscrete {B}arycenters.
\newblock {\em INFORMS Journal on Optimization}, 2:14--33, 2020.

\bibitem{bp-21a}
S.~Borgwardt and S.~Patterson.
\newblock {On the Computational Complexity of Finding a Sparse Wasserstein
  Barycenter}.
\newblock {\em Journal of Combinatorial Optimization}, 41:736–761, 2021.

\bibitem{bp-21}
S.~Borgwardt and S.~Patterson.
\newblock {A Column Generation Approach to the Discrete Barycenter Problem}.
\newblock {\em Discrete Optimization}, 43:100674, 2022.

\bibitem{ce-10}
G.~Carlier and I.~Ekeland.
\newblock Matching for teams.
\newblock {\em Economic Theory}, 42(2):397--418, 2010.

\bibitem{coo-15}
G.~Carlier, A.~Oberman, and E.~Oudet.
\newblock Numerical methods for matching for teams and {W}asserstein
  barycenters.
\newblock {\em ESAIM: Mathematical Modeling and Numerical Analysis},
  49(6):1621--1642, 2015.

\bibitem{cmn-10}
P-A. Chiaporri, R.~McCann, and L.~Nesheim.
\newblock Hedonic price equilibria, stable matching and optimal transport;
  equivalence, topology and uniqueness.
\newblock {\em Economic Theory}, 42(2):317--354, 2010.

\bibitem{cad-20}
S.~Cohen, M.~Arbel, and M.~P. Deisenroth.
\newblock Estimating barycenters of measures in high dimensions.
\newblock {\em eprint arXiv:2007.07105}, 2020.

\bibitem{cfk-13}
C.~Cotar, G.~Friesecke, and C.~Kl\"uppelberg.
\newblock Density functional theory and optimal transportation with coulomb
  cost.
\newblock {\em Communications on Pure and Applied Mathematics}, 66(4):548--599,
  2013.

\bibitem{c-13}
M.~{Cuturi}.
\newblock {Sinkhorn Distances: Lightspeed Computation of Optimal Tansportation
  Distances}.
\newblock In {\em Advances in Neural Information Processing Systems 26},
  number~26, pages 2292--2300, 2013.

\bibitem{cd-14}
M.~Cuturi and A.~Doucet.
\newblock Fast {C}omputation of {W}asserstein {B}arycenters.
\newblock In {\em Proceedings of the 31st International Conference on Machine
  Learning (ICML-14)}, pages 685--693, 2014.

\bibitem{pc-18}
{G. Peyr{\'e} and M. Cuturi}.
\newblock {Computational Optimal Transport}.
\newblock {\em Foundations and Trends in Machine Learning}, 11(5-6):355--607,
  2019.

\bibitem{hbccp-19}
M.~Heitz, N.~Bonneel, D.~Coeurjolly, M.~Cuturi, and G.~Peyr{\'e}.
\newblock {Ground Metric Learning on Graphs}.
\newblock {\em eprint arXiv:1911.03117}, 2019.

\bibitem{hhpj-19}
N.~Ho, V.~Huynh, D.~Phung, and M.~Jordan.
\newblock Probabilistic multilevel clustering via composite transportation
  distance.
\newblock In {\em International Conference on Artificial Intelligence and
  Statistics, PMLR}, pages 3149--3157, 2019.

\bibitem{hnybhp-17}
N.~Ho, X.~L. Nguyen, M.~Yurochkin, H.~H. Bui, V.~Huynh, and D.~Phung.
\newblock Multilevel clustering via {W}asserstein means.
\newblock In {\em International Conference on Machine Learning}, pages
  1501--1509, 2017.

\bibitem{jzd-98}
A.~Jain, Y.~Zhong, and M.-P. Dubuisson-Jolly.
\newblock Deformable template models: {A} review.
\newblock {\em Signal Processing}, 71(2):109--129, 1998.

\bibitem{jbtcg-19}
H.~Janati, T.~Bazeille, B.~Thirion, M.~Cuturi, and A.~Gramfort.
\newblock {Multi-subject MEG/EEG source imaging with sparse multi-task
  regression}.
\newblock {\em eprint arXiv:1910.01914}, 2019.

\bibitem{jcg-20}
H.~Janati, M.~Cuturi, and A.~Gramfort.
\newblock Debiased {S}inkhorn barycenters.
\newblock In {\em International Conference on Machine Learning, PMLR}, page
  4692–4701, 2020.

\bibitem{ktddgu-19}
A.~{Kroshnin}, N.~Tupitsa, D.~{Dvinskikh}, P.~{Dvurechensky}, A.~Gasnikov, and
  C.~Uribe.
\newblock {On the complexity of approximating {W}asserstein barycenters}.
\newblock In {\em International Conference on Machine Learning, PMLR}, pages
  3530--3540, 2019.

\bibitem{lhccj-20}
T.~Lin, N.~Ho, M.~Cuturi, and M.~Jordan.
\newblock {Fixed-support Wasserstein barycenters: Computational hardness and
  fast algorithm}.
\newblock In {\em Advances in Neural Information Processing Systems 33}, 2020.

\bibitem{lspc-10}
G.~Luise, S.~Salzo, M.~Pontil, and C.~Ciliberto.
\newblock Sinkhorn barycenters with free support via {F}rank-{W}olfe algorithm.
\newblock In {\em Advances in Neural Information Processing Systems}, pages
  9322--9333, 2019.

\bibitem{mtbmmh-15}
E.~Munch, K.~Turner, P.~Bendich, S.~Mukherjee, J.~Mattingly, and J.~Harer.
\newblock Probabilistic {F}r\'echet means for time varying persistence
  diagrams.
\newblock {\em Electronic Journal of Statistics}, 9:1173--1204, 2015.

\bibitem{pz-19}
V.~Panaretos and Y.~Zemel.
\newblock {Statistical Aspects of Wasserstein Distances}.
\newblock {\em Annual Review of Statistics and Its Application}, 6(1):405--431,
  2019.

\bibitem{shbmccps-18}
M.~Schmitz, M.~Heitz, N.~Bonneel, F.~Ngol{\'e}, D.~Coeurjolly, M.~Cuturi,
  G.~Peyr{\'e}, and J.-L. Starck.
\newblock {Wasserstein Dictionary Learning: Optimal Transport-Based
  Unsupervised Nonlinear Dictionary Learning}.
\newblock {\em SIAM Journal on Imaging Sciences}, 11(1):643--678, Jan 2018.

\bibitem{swrh-20}
Z.~Shen, Z.~Wang, A.~Riebeiro, and H.~Hassani.
\newblock Sinkhorn barycenters via functional gradient descent.
\newblock In {\em Advances in Neural Information Processing Systems 33}, pages
  986--996, 2020.

\bibitem{sa-19}
D.~Simon and A.~Aberdam.
\newblock Barycenters of natural images - constrained wasserstein barycenters
  for image morphing.
\newblock In {\em 2020 IEEE/CVF Conference on Computer Vision and Pattern
  Recognition}, pages 7907--7916, 2020.

\bibitem{shdj-20}
S.~P. Singh, A.~Hug, A.~Dieuleveut, and M.~Jaggi.
\newblock Context mover's distance \& barycenters: {O}ptimal transport of
  contexts for building representations.
\newblock In {\em International Conference on Artificial Intelligence and
  Statistics, PMLR}, pages 3437--3449, 2020.

\bibitem{sgpcbndg-15}
J.~Solomon, F.~De~Goes, G.~Peyr{\'e}, M.~Cuturi, A.~Butscher, A.~Nguyen, T.~Du,
  and L.~Guibas.
\newblock Convolutional {W}asserstein distances: {E}fficient optimal
  transportation on geometric domains.
\newblock {\em ACM Transactions on Graphics}, 34:1--11, 2015.

\bibitem{ty-05}
A.~Trouv\'e and L.~Younes.
\newblock Local {G}eometry of {D}eformable {T}emplates.
\newblock {\em SIAM Journal on Mathematical Analysis}, 37(1):17--59, 2005.

\bibitem{v-09}
C.~Villani.
\newblock {\em Optimal transport: old and new}, volume 338.
\newblock Springer-Verlag Berlin Heidelberg, 2009.

\bibitem{xwlc-18}
H.~Xu, W.~Wang, W.~Liu, and L.~Carin.
\newblock Distilled {W}asserstein learning for word embedding and topic
  modeling.
\newblock In {\em Advances in Neural Information Processing Systems 32}, pages
  1723--1732, 2018.

\bibitem{ydpbbgc-19}
Y.~Yan, S.~Duffner, P.~Phutane, A.~Berthelier, C.~Blanc, C.~Garcia, and
  T.~Chateau.
\newblock {2D Wasserstein Loss for Robust Facial Landmark Detection}.
\newblock {\em Pattern Recognition}, 116(C):107945, 2021.

\bibitem{ylwwll-17}
J.~Ye, Y.~Li, Z.~Wu, J.~Z. Wang, W.~Li, and J.~Li.
\newblock Determining gains acquired from word embedding quantitatively using
  discrete distribution clustering.
\newblock In {\em Proceedings of the Association for Computational
  Linguistics}, pages 1847--1856, 2017.

\bibitem{zp-17}
Yoav Zemel and Victor~M. Panaretos.
\newblock Fr\'echet means and {P}rocrustes analysis in {W}asserstein space.
\newblock {\em Bernoulli}, 25(2):932--976, 2019.

\end{thebibliography}
\bibliographystyle{plain}

\begin{appendix}
\section{Constraint matrix of MIP (\ref{Gen-MIP})}\label{app:TU}
We exhibit that the constraint matrix of MIP (\ref{Gen-MIP}) is not totally-unimodular. While it is a $0,1$-matrix, it is not hard to verify any of the well-known criteria that rule out total unimodularity.

For example, consider the submatrix corresponding to variables $z_{11}, z_{12}, z_{21}, z_{1211}, z_{1221}$ and to constraints
\begin{align*}
    z_{11} +  z_{12} + \cdots & = 1 \\
-z_{11} +  z_{1211} & \leq 0 \\
-z_{21} +  z_{1211} & \leq 0\\
-z_{12} +  z_{1221} & \leq 0\\
-z_{21} +  z_{1221} & \leq 0 .
\end{align*}
The corresponding coefficient matrix has the form
\[U= 
\begin{pmatrix}
1 & 1  & 0 & 0 & 0\\
-1 & 0 & 0 & 1 & 0 \\
0 & 0 & -1 & 1 & 0 \\
0 & -1 & 0 & 0 & 1\\
0 & 0 & -1 & 0 & 1
\end{pmatrix}, 
\]
which is Eulerian, i.e., it has even row and column sums, but the total sum of elements is 2, and not a multiple of 4. This violates one of the properties of totally-unimodular matrices. Other criteria are easy to check as well. For example, the determinant of $U$ is $\text{det}(U)=-2$. 
\end{appendix}

\end{document}